\documentclass[11pt,a4paper]{amsart}

\usepackage{mathtools}
\usepackage{amsthm}
\usepackage{amssymb}
\usepackage{hyperref}
\usepackage{caption}

 \usepackage{tikz}
 \usetikzlibrary{decorations.pathmorphing}
 \usetikzlibrary{decorations.pathreplacing}

\theoremstyle{definition}
\newtheorem{definition}{Definition}[section]

\newtheorem{example}[definition]{Example}

\theoremstyle{plain}

\newtheorem{lemma}[definition]{Lemma}
\newtheorem{proposition}[definition]{Proposition}
\newtheorem{theorem}[definition]{Theorem}

\numberwithin{equation}{section}

\DeclareMathOperator{\im}{Im}

\tikzset{
  x=10mm,
  y=10mm,
  treenode/.style={
    circle,
    minimum size=1mm,
    inner sep=0,
    fill=gray,
  },
  edgelabel/.style={
    node font=\small,
    inner sep =.5mm,
  },
  columnlabel/.style={
    anchor=north,
  }
}

\tikzset{
  vertex/.style={
    circle,
    minimum size=3mm,
    fill,
    inner sep=0,
    outer sep=0,
  },
  edge/.style={
    line width=.5mm,
  }
}

\allowdisplaybreaks

\begin{document}

\title{Commutative nilpotent transformation semigroups}

\author{Alan J. Cain}
\address[A.J. Cain]{%
Center for Mathematics and Applications (NOVA Math)\\
NOVA School of Science and Technology\\
NOVA University of Lisbon\\
2829--516 Caparica\\
Portugal
}
\email{%
a.cain@fct.unl.pt
}

\author{António Malheiro}
\address[A. Malheiro]{%
Center for Mathematics and Applications (NOVA Math) \& Department of Mathematics\\
NOVA School of Science and Technology\\
NOVA University of Lisbon\\
2829--516 Caparica\\
Portugal
}
\email{%
ajm@fct.unl.pt
}

\author{Tânia Paulista}
\address[T. Paulista]{%
Center for Mathematics and Applications (NOVA Math) \& Department of Mathematics\\
NOVA School of Science and Technology\\
NOVA University of Lisbon\\
2829--516 Caparica\\
Portugal
}
\email{%
t.paulista@campus.fct.unl.pt
}
\thanks{This work is funded by national funds through the FCT -- Fundação para a Ciência e a Tecnologia, I.P., under the scope of the projects UIDB/00297/2020 and UIDP/00297/2020 (Center for Mathematics and Applications)}

\thanks{The third author is funded by national funds through the FCT -- Fundação para a Ciência e a Tecnologia, I.P., under the scope of the studentship 2021.07002.BD}

\subjclass[2020]{Primary 20M20, 20M14; Secondary 05C25, 05C05}

\begin{abstract}
Cameron, et al. determined the maximum size of a null subsemigroup of the full transformation semigroup $\mathcal{T}(X)$ on a finite set $X$ and provided a description of the null semigroups that achieve that size. In this paper we extend the results on null semigroups (which are commutative) to commutative nilpotent semigroups. Using a mixture of algebraic and combinatorial techniques, we show that, when $X$ is finite, the maximum order of a commutative nilpotent subsemigroup of $\mathcal{T}(X)$ is equal to the maximum order of a null subsemigroup of $\mathcal{T}(X)$ and we prove that the largest commutative nilpotent subsemigroups of $\mathcal{T}(X)$ are the null semigroups previoulsy characterized by Cameron, et al..
\end{abstract}

\maketitle

\section{Introduction}

This paper focuses on commutative nilpotent subsemigroups of $\mathcal{T}(X)$, the semigroup of full transformations over $X$. More specifically, when $X$ if finite, we determine the maximum size of a commutative nilpotent subsemigroup of $\mathcal{T}(X)$ and we specify which semigroups achieve this maximum size.

A similar result already exists for null subsemigroups of $\mathcal{T}(X)$ when $X$ is finite \cite{Null_semigroups}. Cameron et al. proved that the maximum size of a null subsemigroup of $\mathcal{T}(X)$ is given by $\max\{t^{|X|-t}:t\in\{1,\ldots,|X|\}\}$ and characterized the null semigroups of maximum order.

Null semigroups are a special case of nilpotent semigroups. Additionally, they are also commutative. For this reason, our main result extends the one about null subsemigroups of $\mathcal{T}(X)$. In fact, we prove that the maximum size of a commutative nilpotent semigroup is equal to the maximum size of a null semigroup. Moreover, we show that the commutative nilpotent semigroups of maximum size are actually the null semigroups described in \cite{Null_semigroups}.

Our results can be compared to the ones obtained by Biggs, Rankin and Reis on nilpotent subsemigroups of $\mathcal{T}(X)$, for a finite set $X$ --- in \cite{Non_commutative_nilpotent_semigroups} they proved that the maximum size of these semigroups is $(|X|-1)!$. We show that, when $|X|\geqslant 4$, non-commutative nilpotent subsemigroups of $\mathcal{T}(X)$ can be much larger than the commutative ones.

Some other authors also tried to find the largest groups and semigroups that satisfy a certain property. For instance, Burns and Goldsmith \cite{Symmetric_group} characterized the largest abelian subgroups of the symmetric group and Vdovin \cite{Alternating_group} characterized the largest abelian subgroups of the alternating group. Gray and Mitchell \cite{Largest_subsemigroups_transformation} determined the maximum order of several subsemigroups of the semigroup $\mathcal{T}_n$ of full transformations over $\{1,\ldots,n\}$, namely the left and right zero semigroups, the completely simple semigroups and the inverse semigroups. Araújo, Bentz and Janusz \cite{Commuting_graph_I(X)} characterized, for a finite set $X$, the largest commutative inverse subsemigroups of the symmetric inverse semigroup $\mathcal{I}(X)$, as well as the largest commutative nilpotent subsemigroups of $\mathcal{I}(X)$. 

This paper is organized in the following way. We begin with Section \ref{Preliminaries}, where we provide some background needed to understand Section \ref{Commutative nilpotent subsemigroups}. This includes some results regarding null subsemigroups of $\mathcal{T}(X)$ of maximum order proved in \cite{Null_semigroups}.

In Section \ref{Commutative nilpotent subsemigroups} we describe the largest commutative nilpotent subsemigroups of $\mathcal{T}(X)$ and determine its order. In the process we also prove that for each commutative nilpotent subsemigroup of $\mathcal{T}(X)$ there is a null subsemigroup of $\mathcal{T}(X)$ of the same size.

\section{Preliminaries} \label{Preliminaries}

Let $\mathcal{T}(X)$ be the semigroup of full transformations on the set $X$. Throughout this paper, $X$ will denote a finite set.

Let $S$ be a semigroup with a zero $0$. We say that $S$ is a {\it nilpotent semigroup} is there exists $m\in\mathbb{N}$ such that the product of any $m$ elements of $S$ is equal to the zero $0$. This is equivalent to write that $S^m=\{0\}$ for some $m\in\mathbb{N}$. If $S^2=\{0\}$ then we say that $S$ is a {\it null semigroup}.

In \cite{Null_semigroups} Cameron et al. introduced two functions $\xi,\alpha:\mathbb{N}\rightarrow\mathbb{N}$ which, for each $n\in\mathbb{N}$, are defined in the following way
\begin{displaymath}(n)\xi=\max\{t^{n-t}:t\in\{1,\ldots,n\}\}\end{displaymath}
and
\begin{displaymath}(n)\alpha=\max\{t\in\{1,\ldots,n\}:t^{n-t}=(n)\xi\}.\end{displaymath}

The next lemma provides some inequalities satisfied by the function $\xi$ described above.

\begin{lemma}[{\cite[Lemma 2.4]{Null_semigroups}}]\label{inequalities xi}
\begin{enumerate}
    \item We have $(1)\xi=(2)\xi$ and \allowbreak $(n)\xi<(n+1)\xi$ for all $n\geqslant 2$;
    \item $(n)\xi(m)\xi\leqslant(n+m-1)\xi$ for all $n,m\in\mathbb{N}$.
\end{enumerate}
\end{lemma}

Theorem~\ref{maximum size null semigroup} shows that the size of a largest null subsemigroup of $\mathcal{T}(X)$ depends on the function $\varepsilon$.

\begin{theorem} [{\cite[Theorem 4.4]{Null_semigroups}}]\label{maximum size null semigroup}
The maximum size of a null subsemigroup of $\mathcal{T}(X)$ is $(|X|)\xi$.
\end{theorem}

Theorem~\ref{null semigroups of maximum size} uses the function $\alpha$ to characterize all the null subsemigroups of $\mathcal{T}(X)$ that achieve the size $(|X|)\xi$ mentioned in the previous theorem.

\begin{theorem} [{\cite[Subsection 4.1]{Null_semigroups}}]\label{null semigroups of maximum size}
Let $S$ be a null subsemigroup of $\mathcal{T}(X)$ such that $|S|=(|X|)\xi$. Let $t=(|X|)\alpha$.
\begin{enumerate}
    \item If the zero of $S$ has rank $1$, then
    \begin{displaymath}
    S=\{\beta \in \mathcal{T}(X): \{x_1,\ldots,x_t\}\beta=\{x_1\} \textrm{ and } \im\beta\subseteq\{x_1,\ldots,x_t\}\}
\end{displaymath}
for some $x_1,\ldots,x_t\in X$.

    \item If the zero of $S$ has rank at least $2$, then $|X|=2$ and the only transformation of $S$ is the identity.
\end{enumerate}

\end{theorem}

\section{Commutative nilpotent subsemigroups of $\mathcal{T}(X)$ of maximum order} \label{Commutative nilpotent subsemigroups}

In this section we investigate commutative nilpotent subsemigroups of $\mathcal{T}(X)$. More specifically, we demonstrate that the maximum order of these semigroups is $(|X|)\xi$ --- the maximum order of a null subsemigroup of $\mathcal{T}(X)$. Additionally, we prove that the largest commutative nilpotent subsemigroups of $\mathcal{T}(X)$ are null semigroups --- the ones described in Theorem~\ref{null semigroups of maximum size}. Furthermore, we show that given a commutative nilpotent subsemigroup of $\mathcal{T}(X)$ we can always find a null subsemigroup of $\mathcal{T}(X)$ of the same size.

%\textcolor{blue}{In order to do so, we present some properties of nilpotent semigroups of $\mathcal{T}(X)$ and define a partition of $X$ using those semigroups. This partition, along with a lemma that arises as a consequence of commutativity, are the key idea to show that }

\begin{proposition}\label{property nilpotent semigroups}
Let $S$ be a nilpotent subsemigroup of $\mathcal{T}(X)$. Assume that $e$ is the zero of $S$. Then
\begin{enumerate}
    \item $x\beta=x$ for all $x\in\im e$ and $\beta\in S$;
    \item $y\beta\in xe^{-1}\setminus\{y\}$ for all $x\in\im e$, $y\in xe^{-1}\setminus\{x\}$ and $\beta\in S$.
\end{enumerate}
\end{proposition}

\begin{proof}
Let $x\in \im e$, $y\in xe^{-1}\setminus\{x\}$ and $\beta\in S$. 

We have $\beta e=e\beta=e$ because $e$ is the zero of $S$. Also, since $x\in\im e$, there exists $z\in X$ such that $ze=x$. Hence $x\beta=ze\beta=ze=x$, which proves (1).

Since $S$ is a nilpotent semigroup, there exists $m \in\mathbb{N}$ such that $S^m=\{e\}$. Then $y\beta^m=ye=x$, and so $y\beta\neq y$. Furthermore, $y\beta e=ye=x$, and so $y\beta\in xe^{-1}$. Consequently, $y\beta\in xe^{-1}\setminus\{y\}$, which proves (2).
\end{proof}

\begin{proposition} \label{exists x not in union of images}
Let $S$ be a nilpotent subsemigroup of $\mathcal{T}(X)$ whose zero has rank $1$. If $|X| \geqslant 2$, then $\bigcup_{\beta\in S} \im \beta \subsetneq X$.
\end{proposition}

\begin{proof}
Let $e$ be the zero of $S$ and let $x\in X$ be such that $\im e=\{x\}$. Since $S$ is nilpotent, there exists $m\in\mathbb{N}$ such that $S^m=\{e\}$.

Suppose, with the aim of obtaining a contradiction, that $\bigcup_{\beta\in S} \im \beta = X$. 

Let $x_1 \in X\setminus\{x\}$. There exist $\beta_1 \in S$ and $x_2 \in X$ such that $x_1=x_2\beta_1$. By Proposition~\ref{property nilpotent semigroups}(1), $x\beta_1=x\neq x_1=x_2\beta_1$, which implies that $x_2\in X \setminus\{x\}$. Continuing in this way, construct a sequence $(x_n)_{n\in\mathbb{N}}$ of elements of $X\setminus\{x\}$ and a sequence $(\beta_n)_{n\in\mathbb{N}}$ of elements of $S$ that verify $x_n=x_{n+1}\beta_n$ for all $n\in\mathbb{N}$. Since $X$ is finite, then there exist $i<j$ such that $x_i=x_j$ and $x_i=x_{i+1}\beta_i=x_{i+2}\beta_{i+1}\beta_i=\ldots=x_j\beta_{j-1}\cdots\beta_{i}=x_i\beta_{j-1}\cdots\beta_{i}$. Consequently, $x_i=x_i\beta_{j-1}\cdots\beta_{i}=x_i(\beta_{j-1}\cdots\beta_{i})^2=\ldots=x_i(\beta_{j-1}\cdots\beta_{i})^m=x_ie=x$, which is a contradiction.

Therefore $\bigcup_{\beta\in S} \im \beta \subsetneq X$.
\end{proof}

\begin{definition}
Let $S$ be a nilpotent subsemigroup of $\mathcal{T}(X)$ whose zero has rank $1$. Let $e$ be the zero of $S$. Given a partition $\{A_j\}_{j=0}^k$ of $X$, we say that $\{A_j\}_{j=0}^k$ is an {\it $S$-partition of $X$} if
\begin{align*} &A_0=\im e \\
&A_j=\left\{x\in X\setminus \bigcup_{l=0}^{j-1}A_l: x\beta\in\bigcup_{l=0}^{j-1} A_l \textrm{ for all } \beta \in S\right\}, \textrm{ } j=1,\ldots,k.\end{align*}
\end{definition}

Note that, from construction, given a nilpotent subsemigroup $S$ of $\mathcal{T}(X)$ whose zero has rank $1$, there is at most one $S$-partition of $X$. We will prove in Proposition~\ref{S-partition} below that an $S$-partition always exists, but first we illustrate the definition with an example.

\begin{example} \label{example S-partition}
We consider the semigroup $\mathcal{T}_6$ of full transformations over $\{1,2,3,4,5,6\}$. Let $S$ be the subsemigroup of $\mathcal{T}_6$ formed by the following transformations:
\begin{displaymath}\left(\begin{array}{cccccc} 1&2&3&4&5&6 \\ 5&5&5&5&5&5\end{array}\right), \left(\begin{array}{cccccc}1&2&3&4&5&6 \\ 5&5&1&5&5&5\end{array}\right), \left(\begin{array}{cccccc}1&2&3&4&5&6 \\ 5&1&2&1&5&1\end{array}\right)\end{displaymath}
\begin{displaymath}
\left(\begin{array}{cccccc}1&2&3&4&5&6 \\ 5&1&4&1&5&1\end{array}\right), \left(\begin{array}{cccccc}1&2&3&4&5&6 \\ 5&1&6&1&5&1\end{array}\right).
\end{displaymath}

Notice that the first transformation is the zero of the semigroup and has rank $1$. Furthermore, the product of any three transformations is equal to the zero of $S$. It is straightforward to verify that $S$ is a commutative semigroup. Hence $S$ is a commutative nilpotent semigroup.

We are going to determine the $S$-partition of $\{1,2,3,4,5,6\}$. The set $A_0$ is equal to the image of the zero of $S$, which implies that $A_0=\{5\}$. The set $A_1$ is formed by all the elements of $\{1,2,3,4,5,6\}\setminus A_0=\{1,2,3,4,6\}$ whose image in all the transformations of $S$ belongs to $A_0$, that is, whose image is always $5$. Since $1$ is the only element with that property, we have $A_1=\{1\}$. The set $A_2$ is formed by all the elements of $\{1,2,3,4,5,6\}\setminus (A_0\cup A_1)=\{2,3,4,6\}$ whose image in the transformations of $S$ always belongs to $A_0\cup A_1$, that is, whose image is either $1$ or $5$. Hence $A_2=\{2,4,6\}$. The set $A_3$ is formed by the remaining element of $\{1,2,3,4,5,6\}$, that is, $A_3=\{3\}=\{1,2,3,4,5,6\}\setminus (A_0\cup A_1\cup A_2)$. Notice that the image of $3$ in the transformations of $S$ always belongs to $A_0\cup A_1\cup A_2=\{1,2,4,5,6\}$. Since $A_0\cup A_1\cup A_2\cup A_3=\{1,2,3,4,5,6\}$, then $\{A_j\}_{j=0}^3$ is the $S$-partition of $\{1,2,3,4,5,6\}$.
\end{example}

\begin{proposition}\label{S-partition}
Let $S$ be a nilpotent subsemigroup of $\mathcal{T}(X)$ whose zero has rank $1$. Then there exists an $S$-partition of $X$. 
\end{proposition}

\begin{proof}
Let $e$ be the zero of $S$ and $i\in X$ be such that $\im e=\{i\}$. 

Let $n=|X|$. We are going to prove the result by induction on $n$.

Suppose that $n=1$. Then $S=\left\{\left(\begin{array}{c} i\\ i\end{array}\right)\right\}=\{e\}$ and $X=\{i\}=\im e$. Thus $\{\im e\}$ is an $S$-partition of $X$.

Suppose that $n\geqslant 2$ and assume that the result is valid for $n-1$.

Since $S$ is a nilpotent semigroup then, by Proposition~\ref{exists x not in union of images}, there exists $t\in X\setminus\bigcup_{\beta\in S}\im \beta$, which implies that $\beta|_{X\setminus\{t\}}\in\mathcal{T}(X\setminus\{t\})$ for all $\beta\in S$. It is easy to see that $S'=\{\beta|_{X\setminus\{t\}}: \beta\in S\}$ is a nilpotent subsemigroup of $\mathcal{T}(X\setminus\{t\})$ whose zero is $e|_{X\setminus\{t\}}$. The rank of $e|_{X\setminus\{t\}}$ is also $1$ and $\im e|_{X\setminus\{t\}}=\{i\}=\im e$. By the induction hypothesis, $X\setminus\{t\}$ admits an $S'$-partition $\{A_j\}_{j=0}^{k}$, where $A_0=\im e|_{X\setminus\{t\}}=\im e$ and, for all $j\in\{1,\ldots,k\}$,
\begin{align*}
A_j&=\left\{x\in (X\setminus \{t\})\setminus \bigcup_{l=0}^{j-1}A_l: x\beta\in\bigcup_{l=0}^{j-1} A_l \textrm{ for all } \beta \in S'\right\}\\
&=\left\{x\in (X\setminus \{t\})\setminus \bigcup_{l=0}^{j-1}A_l: x\beta|_{X\setminus\{t\}}\in\bigcup_{l=0}^{j-1} A_l \textrm{ for all } \beta \in S\right\}\\
&=\left\{x\in (X\setminus \{t\})\setminus \bigcup_{l=0}^{j-1}A_l: x\beta\in\bigcup_{l=0}^{j-1} A_l \textrm{ for all } \beta \in S\right\}.\end{align*}

From the definition of $t$, we have $t\neq i$ and $t\beta\neq t$ for all $\beta\in S$. Let
\begin{displaymath} r=\min\left\{s\in\{1,\ldots,k+1\}: t\beta\in\bigcup_{l=0}^{s-1}A_l \textrm{ for all } \beta \in S\right\}.\end{displaymath} 

We want to construct an $S$-partition of $X$ from the $S'$-partition $\{A_j\}_{j=0}^k$ of $X\setminus\{t\}$. We will either create a new set $A_{k+1}$ formed exclusively by $t$, or add $t$ to one of the existing sets of $\{A_j\}_{j=0}^k$. The way we extend the partition of $X\setminus\{t\}$ depends on the value of $r$ defined above and is chosen so that the new partition is an $S$-partition of $X$.

We consider two cases. First, suppose that $r=k+1$. This implies that there exists $\beta \in S$ such that $t\beta \notin \bigcup_{l=0}^{k-1}A_l$. Consequently, 
\begin{displaymath}A_j=\left\{x\in X\setminus \bigcup_{l=0}^{j-1}A_l: x\beta\in\bigcup_{l=0}^{j-1} A_l \textrm{ for all } \beta \in S\right\}\end{displaymath}
for all $j\in\{1,\ldots,k\}$. Let $A_{k+1}=\{t\}$. Hence 
\begin{displaymath}A_{k+1}=\left\{x\in X\setminus \bigcup_{l=0}^{k}A_l: x\beta\in\bigcup_{l=0}^{k} A_l \textrm{ for all } \beta \in S\right\}\end{displaymath}
and $\{A_j\}_{j=0}^{k+1}$ is an $S$-partition of $X$.

Suppose that $r\leqslant k$. Let $A'_r=A_r\cup\{t\}$ and $A'_j=A_j$ for all $j\in\{0,\ldots,k\}\setminus\{r\}$. We also have $t\beta \notin \bigcup_{l=0}^{r-2}A'_l$ for some $\beta \in S$. Then
\begin{displaymath}A'_j=\left\{x\in X\setminus \bigcup_{l=0}^{j-1}A'_l: x\beta\in\bigcup_{l=0}^{j-1} A'_l \textrm{ for all } \beta \in S\right\}\end{displaymath}
for all $j\in\{1,\ldots,k\}$. Thus $\{A'_j\}_{j=0}^{k}$ is an $S$-partition of $X$.
\end{proof}

The concept of $S$-partition plays a key role in the proof of Lemma~\ref{main lemma}, which explains how commutativity restricts the structure of the maps of $S$. This result is the key idea used to determine the maximum size of a commutative nilpotent subsemigroup of $\mathcal{T}(X)$.

\begin{lemma}\label{main lemma}
Let $S$ be a commutative nilpotent subsemigroup of $\mathcal{T}(X)$ whose zero has rank $1$, and $\{A_j\}_{j=0}^{k}$ be the $S$-partition of $X$. Let $i\in\{1,\ldots,k\}$ and define $A=\bigcup_{j=0}^{i-1}A_j$. Let $x\in A_i$ and $\beta_1,\ldots,\beta_m \in S$ be such that $\beta_1|_A=\cdots=\beta_m|_A$. Then $(x\beta_1)\gamma=\cdots=(x\beta_m)\gamma$ for all $\gamma \in S$.
\end{lemma}

\begin{proof}
Let $\gamma \in S$ and $l,t\in\{1,\ldots,m\}$. Since $x\in A_i$, then $x\gamma \in \bigcup_{j=0}^{i-1}A_j=A$. Hence, since $S$ is commutative, we have $(x\beta_l)\gamma=(x\gamma)\beta_l=(x\gamma)\beta_l|_A=(x\gamma)\beta_t|_A =(x\gamma)\beta_t=(x\beta_t)\gamma$.
\end{proof}

\begin{theorem}\label{|S|<=xi(n), zero rank 1}
Let $S$ be a commutative nilpotent subsemigroup of $\mathcal{T}(X)$ whose zero has rank $1$, and let $\{A_j\}_{j=0}^{k}$ be the $S$-partition of $X$. Then there exists a null subsemigroup $N$ of $\mathcal{T}(X)$ such that $|S|=|N|$.
\end{theorem}

\begin{proof}
The idea of the proof is to construct a labelled tree using the semigroup $S$, modify the tree and obtain a new one which will be the labelled tree of a null semigroup of size $|S|$. In order to obtain a labelled tree from $S$, we first associate each transformation of $S$ to a word of length $|X|$ over $X$, and then use these words to create a tree. Lemma~\ref{main lemma} is what assures that we can modify the tree the way we do.  Finally, we can obtain a new set of words from the new tree and see that the transformations associated to those words form a null semigroup. (For an illustration of how the proof applies to a particular semigroup, see Example~\ref{example proof}.)

Let $n=|X|$.

For simplicity, we start by reordering the elements of $X$ in a way such that the elements of $A_j$ appear before the elements of $A_{j+1}$ for all $j\in\{0,\ldots,k-1\}$. Assume that, after reordering, the elements of $X$ are sequenced in the following way: $x_1,\ldots,x_n$.

Using the order $x_1,\ldots,x_n$ and the semigroup $S$, we are going to create a set of words over $X$. Each transformation $\beta\in S$ determines the word $w_\beta$ of length $n$ over $X$ whose $i$-th letter is $x_i\beta$. Let $W=\{w_{\beta}: \beta\in S\}$ be the set of words determined by $S$. Note that $|W|=|S|$.

We are going to construct a labelled tree using the set of words $W$. The set of vertices of the tree is the set of prefixes of the words belonging to $W$, that is, the set of vertices is $\{u\in X^*: uv \in W \textrm{ for some } v\in X^*\}$. Each arc is labelled with a letter from the alphabet $X$ and, given two vertices $u$ and $v$, we have an arc from $u$ to $v$ labelled by the letter $x$ if and only if $ux=v$.

Observe that, as a consequence of the way we defined the tree, the vertex $\varepsilon$ is the only one whose indegree is zero, which means that the vertex $\varepsilon$ is the root of the tree.

Let $\beta \in S$ and assume that $w_\beta=w_1\cdots w_n$, where $w_1,\ldots,w_n\in\{x_1,\ldots,\allowbreak x_n\}$. If we start at the vertex $\varepsilon$ and we follow the path formed by the arcs labelled by $w_1,\ldots,w_n$, we end up at the vertex $w_\beta=w_1\cdots w_n$, which corresponds to a leaf of the tree. Hence each one of the leaves is associated with a unique word of $W$ (and, consequently, a unique transformation of $S$). Therefore the tree has $|S|$ leaves.

Since $(A_0\cup A_1)\beta=\{x_1\}$ for all $\beta \in S$, then $x_1, x_1^2, \ldots, x_1^{|A_0\cup A_1|}$ are prefixes of all the words in $W$ and, consequently, for each $i\in\{1,\ldots,|A_0\cup A_1|\}$ $x_1^i$ is the only vertex of length $i$. In the tree, this translates into a path of length $|A_0\cup A_1|$ that begins at the vertex $\varepsilon$ (the root of the tree) and ends at the vertex $x_1^{|A_0\cup A_1|}$, and where all the arcs have label $x_1$. We call this path the {\it trunk} of the tree. Since $|A_0|=1$, the length of the trunk is at least $1$. All the paths labelled by some word $w_\beta \in W$ starting at the root (vertex $\varepsilon$) and ending at a leaf (the one associated with $\beta$) contain the trunk of the tree. 

Assume that we have at least two transformations $\beta_1,\beta_2$ of $S$ that are equal in $\{x_1,\ldots,x_i\}$ and different in $\{x_{i+1}\}$ (meaning that $x_{i+1}\beta_1\neq x_{i+1}\beta_2$) for some $i\in\{1,\ldots,n-1\}$. Then the words $w_{\beta_1},w_{\beta_2}$ determined by $\beta_1,\beta_2$ share the same prefixes of length between $0$ and $i$, and the prefix of length $i+1$ is different for each one of the words. Hence, in the tree of $S$, the paths labelled by $w_{\beta_1},w_{\beta_2}$ starting at the root and ending at the leaves $w_{\beta_1},w_{\beta_2}$ coincide on the first $i$ arcs and diverge on the $(i+1)$-th arc. This means that the starting vertex of the $(i+1)$-th arcs of these paths has outdegree at least $2$. In cases like this one, where we have a vertex with outdegree at least $2$, that is, a vertex that has at least $2$ arcs starting in it, we say that a {\it branching} occurs.

Let $i\in\{1,\ldots,n\}$. Consider all the arcs of the tree whose starting vertex is a word of length $i-1$ and ending vertex is a word of length $i$. We say that those arcs form the {\it level} $x_i$ of the tree. If there is at least one branching at level $x_i$ (which happens if there exists a vertex that is a word of length $i-1$ whose outdegree is at least $2$), then we call it a {\it branching level}. If no branching occurs at level $x_i$ (that is, if all the vertices that are words of length $i-1$ have outdegree $1$), then we call it a {\it linear level}.

Note that the levels of the trunk of the tree (the first $|A_0\cup A_1|$ levels) are all linear and contain exactly one arc.

Given an arc of the level $x_i$, its label corresponds to the $i$-th letter of some word $w_\beta$ of $W$, which is equal to $x_i\beta$. Notice that if $i>1$ then, because of the way we ordered the elements of $X$, the labels of the arcs of the level $x_i$ belong to $\{x_1,\ldots,x_{i-1}\}$.

The next two lemmata are a consequence of Lemma~\ref{main lemma} and provide some properties of the tree of $S$ that relate the notions of branching and linear level. Lemma~\ref{i_1,...,i_s<i, i_2,...,i_s linear levels} allow us to modify the tree of $S$ and Lemma~\ref{branching preceded by s linear levels} guarantees that the new tree can be that of a null semigroup.

\begin{lemma}\label{i_1,...,i_s<i, i_2,...,i_s linear levels}
Suppose that we have a branching at level $x_i$ for some $i\in\{2,\ldots,n\}$. Let $s\geqslant 2$ be the number of arcs in that branching, and $x_{i_1},\ldots, x_{i_s}$ be the labels of those arcs (where $i_1<i_2<\cdots<i_s$). Then $i_s<i$ and the levels $x_{i_2},\ldots,x_{i_s}$ are linear.
\end{lemma}

\begin{proof}
The labels of the arcs of the level $x_i$ belong to $\{x_1,\ldots,x_{i-1}\}$. Hence $x_{i_s}\in\{x_1,\ldots,x_{i-1}\}$ and, consequently, $i_s<i$.

The existence of a branching at level $x_i$ with $s$ arcs, whose labels are $x_{i_1},\ldots,x_{i_s}$, implies the existence of $s$ transformations of $S$, $\beta_1,\ldots,\beta_s$, that are equal in $\{x_1,\ldots,x_{i-1}\}$ and such that $x_i\beta_1,\ldots,x_i\beta_s \in \{x_{i_1},\ldots,x_{i_s}\}$ and are pairwise distinct. Assume, without loss of generality, that $x_i\beta_j=x_{i_j}$ for all $j\in\{1,\ldots,s\}$.

Let $l\in\{1,\ldots,k\}$ be such that $x_i\in A_l$. Because of the way we ordered the elements of $X$, $\bigcup_{j=0}^{l-1}A_j\subseteq \{x_1,\ldots,x_{i-1}\}$. Hence $\beta_1,\ldots,\beta_s$ are equal in $\bigcup_{j=0}^{l-1}A_j$.

We want to see that the levels $x_{i_2},\ldots,x_{i_s}$ are linear. Let $j\in\{2,\ldots,s\}$. Let $u$ be a vertex that is a word of length $i_j-1$. Then $u$ is the starting vertex of some arc of the level $x_{i_j}$. Choose one of the arcs whose starting vertex is $u$ and let $\gamma\in S$ be a transformation whose path starting at the root, ending at a leaf and labelled by $w_\gamma$ contains the chosen arc. Then the label of that arc is equal to $x_{i_j}\gamma$ and the label of the arc of the level $x_{i_1}$ belonging to that path is equal to $x_{i_1}\gamma$. By Lemma~\ref{main lemma}, $x_i\beta_1\gamma=x_i\beta_j\gamma$ and, consequently, $x_{i_1}\gamma=x_{i_j}\gamma$. Therefore the only arc with starting vertex $u$ is the one with label $x_{i_1}\gamma$. Thus $u$ has outdegree $1$.

We just proved that all the starting vertices of the arcs of the level $x_{i_j}$ have outdegree $1$. Thus the level $x_{i_j}$ is linear. Since $j$ is an arbitrary element of $\{2,\ldots,s\}$, then the levels $x_{i_2},\ldots,x_{i_s}$ are all linear.
\end{proof}

As a consequence of Lemma~\ref{i_1,...,i_s<i, i_2,...,i_s linear levels}, we have that a branching with $s$ arcs is associated to $s$ levels that precede it: the first one can either be a linear or a branching level and the last $s-1$ are all linear levels.

\begin{lemma}\label{branching preceded by s linear levels}
A branching with $s$ arcs is preceded by at least $s$ linear levels.
\end{lemma}

\begin{proof}
Assume that we have a branching at level $x_i$ whose arcs have labels $x_{i_1},\ldots,x_{i_s}$ (where $i_1<i_2<\cdots<i_s<i$). By Lemma~\ref{i_1,...,i_s<i, i_2,...,i_s linear levels}, the levels $x_{i_2},\ldots,x_{i_s}$ are all linear and precede the branching. Thus the branching is preceded by at least $s-1$ linear levels.

Suppose that the levels $x_{i_2},\ldots,x_{i_s}$ are all outside of the trunk of the tree. Since the trunk of the tree has at least one level, and all the levels of the trunk are linear and antecede the branching, then the branching is preceded by at least $s$ linear levels.

Suppose that there is at least one level, among the levels $x_{i_2},\ldots,x_{i_s}$, that belongs to the trunk of the tree. Then the level $x_{i_1}$ is also part of the trunk, which implies that the level $x_{i_1}$ contains only one arc whose label is $x_1$.

Let $u$ be an arc of the level $x_{i_s}$. There exists $\gamma\in S$ such that the path labelled by $w_\gamma$, starting at the root and ending at a leaf, contains $u$, which has label $x_{i_s}\gamma$. This path also contains the only arc from the level $x_{i_1}$ whose label is $x_1$. Since the level $x_{i_1}$ is a part of the trunk, then $x_{i_1}\in A_0\cup A_1$, which implies that $x_{i_1}\gamma=x_1$.

Let $l\in\{1,\ldots,k\}$ be such that $x_i\in A_l$. Let $\beta_1,\beta_s\in S$ be such that $\beta_1$ and $\beta_s$ are equal in $\{x_1,\ldots,x_{i-1}\}$, $x_i\beta_1=x_{i_1}$ and $x_i\beta_s=x_{i_s}$. These transformations exist because $x_{i_1}$ and $x_{i_s}$ are labels of arcs that are part of a branching at level $x_i$. We have $\bigcup_{j=0}^{l-1}A_l\subseteq \{x_1,\ldots,x_{i-1}\}$ because of the way we rearranged the elements of $X$ and, as a consequence, $\beta_1$ and $\beta_s$ are equal in $\bigcup_{j=0}^{l-1}A_l$. Then, by Lemma~\ref{main lemma}, $x_1=x_{i_1}\gamma=x_i\beta_1\gamma=x_i\beta_s\gamma=x_{i_s}\gamma$ and $u$ has label $x_1$. Since $u$ is an arbitrary arc of the level $x_{i_s}$, then all the arcs of the level $x_{i_s}$ are labelled by $x_1$. Hence $x_{i_s}\beta=x_1$ for all $\beta\in S$ and, consequently, $x_{i_s}\in A_0\cup A_1$. Thus the level $x_{i_s}$ is part of the trunk and, since the levels $x_{i_2},\ldots,x_{i_{s-1}}$ precede the level $x_{i_s}$, they are also part of the trunk. Therefore the trunk has at least $s$ linear levels and the branching is preceded by at least $s$ linear levels.
\end{proof}

With the aim of finding a null semigroup of size $|S|$, we are going to modify the tree of $S$ and obtain a new tree. We will then see that this tree is that of a null semigroup of size $|S|$. In order to guarantee that the new semigroup has size $|S|$, we make sure that the modifications we apply to the tree of $S$ do not change its number of leaves (which is equal to $|S|$).

We start by deleting the labels of the arcs. Then we consider all the linear levels that do not correspond to the trunk of the tree. Assume that there are $m$ linear levels in the tree, $m'$ of which are the linear levels outside of the trunk. Then $m$ is equal to the sum of $m'$ and the number of arcs in the trunk. We are going to move those $m'$ linear levels to the trunk of the tree, that is, we are going to eliminate all the arcs that correspond to those levels, and we are going to add $m'$ arcs to the trunk of the tree (that is, we are adding $m'$ linear levels to the trunk). Of course, if the tree of $S$ has all its linear levels in the trunk, then we do not need to perform any changes in the tree. Note that, since all the starting vertices of the arcs belonging to the linear levels  have outdegree $1$, then eliminating linear levels does not cause any problems in the tree. This entire process does not change neither the number of leaves of the tree, neither the number of linear and branching levels of the tree. Furthermore, these transformations do not create new branchings and maintain the number of arcs of the existing ones. This means that each branching of the new tree was also a branching of the tree of $S$ (and it has the same number of arcs).

Now we just need to add labels to the arcs and rename the vertices of the new tree, in a way that guarantees that this is a tree of a semigroup. We start by labelling the arcs. All the $m$ arcs belonging to the trunk of the tree are labelled by $x_1$. We now consider the starting vertices of the arcs that do not belong to the trunk of the new tree. We want to label these arcs using exclusively elements from $\{x_1,\ldots,x_m\}$. If we have a vertex with outdegree $1$ then we label the corresponding arc by $x_1$. Assume now that we have a vertex with outdegree $s\geqslant 2$. Then we have a branching at that vertex. According to Lemma~\ref{branching preceded by s linear levels}, in the tree of $S$ this branching was preceded by $s$ linear levels, which are all part of the trunk of the new tree. Hence $s\leqslant m$ and we label the arcs of this branching by $x_1,\ldots,x_s$.

Finally, we rename the vertices. We want the vertices to be the prefixes of the words associated with the leaves, which should be words of length $n$. Hence the root of the tree needs to be the word $\varepsilon$. We also want to guarantee that, given two vertices $u$ and $v$, there is an arc labelled by $x$ from $u$ to $v$ if and only if $v=ux$. Hence the vertices that are not the root must be given by $wx$, where $x$ is the label of the only arc that ends at the vertex we are considering and $w$ is the starting vertex of that arc.

Let $Z$ be the set of words formed by the labels of the leaves of the new tree. Note that we have $|S|$ words, all of which have length $n$. Using again the order $x_1,\ldots,x_n$ of the elements of $X$, we are going to obtain from each word of $Z$ a transformation of $\mathcal{T}(X)$. Let $w=w_1\cdots w_n\in Z$ (where $w_1,\ldots,w_n \in \{x_1,\ldots,x_n\}$). Then $w$ determines the transformation $\beta\in \mathcal{T}(X)$ such that $x_i\beta=w_i$. Let $N$ be the set formed by the transformations obtained from $Z$. We want to prove that $N$ is a null semigroup. First, we notice that $x_1^n\in Z$. Hence the constant map $e$ with image $\{x_1\}$ belongs to $N$. Let $\beta,\gamma\in N$ and $x\in X$. Since the labels of the arcs of the new tree belong to $\{x_1,\ldots,x_m\}$, then $Z\subseteq \{x_1,\ldots,x_m\}^*$ and, consequently, $x\beta\in\{x_1,\ldots,x_m\}$. However, at the trunk of the new tree, the arcs are all labelled $x_1$, which implies that $x_1^m$ is a prefix of all the words in $Z$. Therefore $\{x_1,\ldots,x_m\}\gamma=\{x_1\}$ and, as a consequence, $x\beta\gamma=x_1$. Thus $\beta\gamma=e$.

Therefore $N$ is a null subsemigroup of $\mathcal{T}(X)$ such that $|N|=|Z|=|S|$.
\end{proof}

According to Theorem~\ref{|S|<=xi(n), zero rank 1}, given a commutative nilpotent subsemigroup $S$ of $\mathcal{T}(X)$ whose zero has rank $1$, there exists a null subsemigroup $N$ of $\mathcal{T}(X)$ that has size $|S|$. However, by Theorem~\ref{maximum size null semigroup}, $N$ has at most $(|X|)\xi$ transformations, which means that $S$ also has at most $(|X|)\xi$ transformations. Thus the maximum size of a commutative nilpotent subsemigroup of $\mathcal{T}(X)$ is $(|X|)\xi$.

\begin{example}\label{example proof}
The present example serves as a way to show how the proof of Theorem~\ref{|S|<=xi(n), zero rank 1} works. In order to do that, we consider again the semigroup $S$ from Example~\ref{example S-partition}. We saw that $\{A_j\}_{j=0}^3$ is the $S$-partition of $\{1,2,3,4,5,6\}$, where $A_0=\{5\}$, $A_1=\{1\}$, $A_2=\{2,4,6\}$ and $A_3=\{3\}$.

We want to choose an order of the elements of $\{1,2,3,4,5,6\}$ such that the element of $A_0$ is the first one in that order, the second one is the element of $A_1$, followed by the three elements of $A_2$ (in any order) and the last element is the one belonging to $A_3$. A possible way of ordering the elements is $5,1,2,4,6,3$.

The set of words we obtain from the transformations of $S$ (using the order $5,1,2,4,6,3$) is
\begin{displaymath}W=\{555555, 555551, 551112, 551114, 551116\}\subseteq \{1,2,3,4,5,6\}^*,\end{displaymath}
which allow us to construct the tree in Figure~\ref{tree of S}.

\begin{figure}[hbt]
\begin{center}
    \begin{tikzpicture}[x=8mm,y=8mm]

    %\draw[color=gray!50, dashed, thick](5.5,1.5) circle (0.6);
    \draw[color=gray!50, dashed, thick](4.9,2) rectangle (6.1,1);
    
    \draw[color=gray!50, dashed, thick](4.9,-0.8) rectangle (6.1,-2.2);
    \draw[color=gray!50, dashed, thick](1.9,1.6) rectangle (3.1,-1.6);

    \draw[decorate,decoration={brace, amplitude=6pt, mirror}] (0,-1) -- (2,-1);
    \node at (1,-1.5) {\small Trunk};

    \draw[decorate,decoration={brace, amplitude=6pt, mirror}] (5,-3) -- (6,-3);
    \node at (5.5,-3.5) {\small Branching};

    \draw[decorate,decoration={brace, amplitude=6pt, mirror}] (2,-3) -- (3,-3);
    \node at (2.5,-3.5) {\small Branching};

    \draw[decorate,decoration={brace, amplitude=6pt}] (5,3) -- (6,3);
    \node at (5.5,3.5) {\small Branching};

    % Draw all the nodes
    \begin{scope}[
      every node/.style={treenode},
      ]
      \node (root) at (0,0) {};
      \node (1) at (1,0) {};
      \node (11) at (2,0) {};
      \node (111) at (3,1.5) {};
      \node (1111) at (4,1.5) {};
      \node (11111) at (5,1.5) {};
      \node (111111) at (6,1.75) {};
      \node (111112) at (6,1.25) {};
      \node (112) at (3,-1.5) {};
      \node (1122) at (4,-1.5) {};
      \node (11222) at (5,-1.5) {};
      \node (112223) at (6,-1) {};
      \node (112224) at (6,-1.5) {};
      \node (112225) at (6,-2) {};
    \end{scope}

    \node[anchor=east] at (root) {\small $\varepsilon$};
    \node[anchor=south] at (1) {\rotatebox{90}{\small $5$}};
    \node[anchor=south] at (11) {\rotatebox{90}{\small $55$}};
    \node[anchor=south] at (111) {\rotatebox{90}{\small $555$}};
    \node[anchor=south] at (1111) {\rotatebox{90}{\small $5555$}};
    \node[anchor=south] at (11111) {\rotatebox{90}{\small $55555$}};
    \node[anchor=north] at (112) {\rotatebox{90}{\small $551$}};
    \node[anchor=north] at (1122) {\rotatebox{90}{\small $5511$}};
    \node[anchor=north] at (11222) {\rotatebox{90}{\small $55111$}};
    \node[anchor=west] at (111111) {$555555$};
    \node[anchor=west] at (111112) {$555551$};
    \node[anchor=west] at (112223) {$551112$};
    \node[anchor=west] at (112224) {$551114$};
    \node[anchor=west] at (112225) {$551116$};

    % Draw all the edges
    \begin{scope}[
      ->,
      every node/.style={edgelabel},
      ]
      \draw (root) -- node[above] {$5$} (1);
      \draw (1) -- node[above] {$5$} (11);
      \draw (11) -- node[anchor=south east] {$5$} (111);
      \draw (111) -- node[above] {$5$} (1111);
      \draw (1111) -- node[above] {$5$} (11111);
      \draw (11111) -- node[above] {$5$} (111111);
      \draw (11111) -- node[below] {$1$} (111112);
      \draw (11) -- node[anchor=south west] {$1$} (112);
      \draw (112) -- node[above] {$1$} (1122);
      \draw (1122) -- node[above] {$1$} (11222);
      \draw (11222) -- node[anchor=south east] {$2$} (112223);
      \draw (11222) -- node[pos=.75,above] {$4$} (112224);
      \draw (11222) -- node[anchor=north east] {$6$} (112225);
    \end{scope}

    % Draw the columns first
    \foreach \x in {8.6,9.6,10.6,11.6,12.6,13.6,14.6} {
      \draw[black,dashed] (\x,2.25) --  (\x,-2.5);
    }
    \foreach \x/\t in {8.6/L,9.6/L,10.6/B,11.6/L,12.6/L,13.6/B} {
      \node[columnlabel] at ({\x+.5},-2.5) {\t};
    }
    \foreach \x/\t in {8.6/5,9.6/1,10.6/2,11.6/4,12.6/6,13.6/3} {
      \node[columnlabel] at ({\x+.5},2.75) {\t};
    }
    \node at (11.6,3.25) {\small Levels};
    \node at (11.6,-3.5) {\small Linear (L) and Branching (B) Levels};

    % Draw all the nodes
    \begin{scope}[
      every node/.style={treenode},
      ]
      \node (root) at (8.6,0) {};
      \node (1) at (9.6,0) {};
      \node (11) at (10.6,0) {};
      \node (111) at (11.6,1.5) {};
      \node (1111) at (12.6,1.5) {};
      \node (11111) at (13.6,1.5) {};
      \node (111111) at (14.6,1.75) {};
      \node (111112) at (14.6,1.25) {};
      \node (112) at (11.6,-1.5) {};
      \node (1122) at (12.6,-1.5) {};
      \node (11222) at (13.6,-1.5) {};
      \node (112223) at (14.6,-1) {};
      \node (112224) at (14.6,-1.5) {};
      \node (112225) at (14.6,-2) {};
    \end{scope}

    % Draw all the edges
    \begin{scope}[
      ->,
      every node/.style={edgelabel},
      ]
      \draw (root) -- (1);
      \draw (1) -- (11);
      \draw (11) -- (111);
      \draw (111) -- (1111);
      \draw (1111) -- (11111);
      \draw (11111) -- (111111);
      \draw (11111) -- (111112);
      \draw (11) -- (112);
      \draw (112) -- (1122);
      \draw (1122) -- (11222);
      \draw (11222) -- (112223);
      \draw (11222) -- (112224);
      \draw (11222) -- (112225);
    \end{scope}
  \end{tikzpicture}
\caption{Tree of $S$.}
\label{tree of S}
\end{center}
\end{figure}
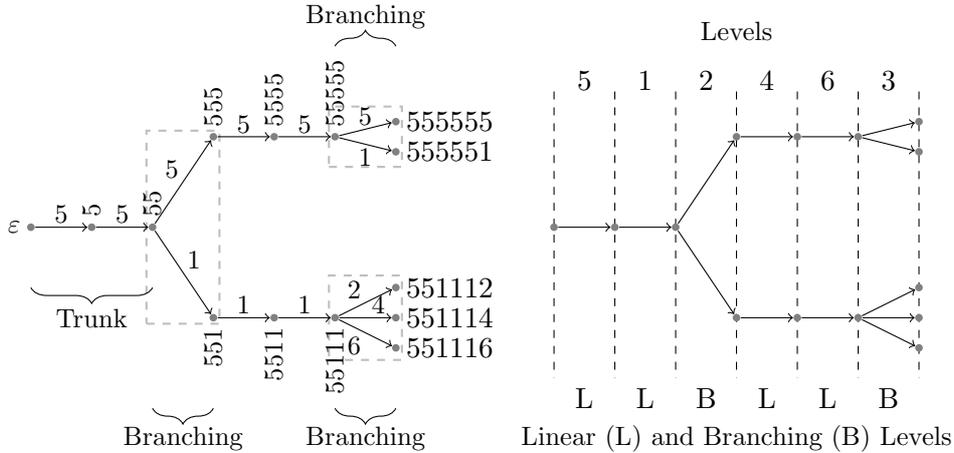

In Figure~\ref{tree of S}, the image on the left is the tree of $S$ with the arcs and vertices labelled. The trunk and branchings of the tree are also identified. The trunk of the tree corresponds to the path from the root (vertex $\varepsilon$) to the vertex $55$. The branchings of the tree are the ones associated with the vertices $55$, $55555$ and $55111$, and are marked by dashed rectangles.

Also in Figure~\ref{tree of S}, in the image on the right the levels of the tree of $S$ are indicated at the top of the tree, and at the bottom are distinguished the linear and branching levels of the tree.

Now we are going to perform some transformations in tree of $S$ in order to obtain a new tree. We consider the two linear levels that are not a part of the trunk of the tree of $S$, namely the levels $4$ and $6$. The idea is to remove those linear levels from the tree (that is, we are going to delete the arcs belonging to levels $4$ and $6$ --- the ones in bold in the tree on the left in Figure~\ref{transforming the tree of S}), and then add two linear levels to the trunk of the tree (that is, we are going to add two arcs to the trunk --- the ones in bold in the tree on the right in Figure~\ref{transforming the tree of S}).

\begin{figure}[hbt]
\begin{center}
\begin{tikzpicture}[x=8mm, y=8mm]

% Draw the columns first
    \foreach \x in {0,1,2,3,4,5,6} {
      \draw[black,dashed] (\x,2.25) --  (\x,-2.5);
    }
    \foreach \x/\t in {0/5,1/1,2/2,3/4,4/6,5/3} {
      \node[columnlabel] at ({\x+.5},2.75) {\t};
    }
    \foreach \x/\t in {0/L,1/L,2/B,3/L,4/L,5/B} {
      \node[columnlabel] at ({\x+.5},-2.5) {\t};
    }

    % Draw all the nodes
    \begin{scope}[
      every node/.style={treenode},
      ]
      \node (root) at (0,0) {};
      \node (1) at (1,0) {};
      \node (11) at (2,0) {};
      \node (111) at (3,1.5) {};
      \node (1111) at (4,1.5) {};
      \node (11111) at (5,1.5) {};
      \node (111111) at (6,1.75) {};
      \node (111112) at (6,1.25) {};
      \node (112) at (3,-1.5) {};
      \node (1122) at (4,-1.5) {};
      \node (11222) at (5,-1.5) {};
      \node (112223) at (6,-1) {};
      \node (112224) at (6,-1.5) {};
      \node (112225) at (6,-2) {};
    \end{scope}

    % Draw all the edges
    \begin{scope}[
      ->,
      every node/.style={edgelabel},
      ]
      \draw (root) -- (1);
      \draw (1) -- (11);
      \draw (11) -- (111);
      \draw[line width=.5mm] (111) -- (1111);
      \draw[line width=.5mm] (1111) -- (11111);
      \draw (11111) -- (111111);
      \draw (11111) -- (111112);
      \draw (11) -- (112);
      \draw[line width=.5mm] (112) -- (1122);
      \draw[line width=.5mm] (1122) -- (11222);
      \draw (11222) -- (112223);
      \draw (11222) -- (112224);
      \draw (11222) -- (112225);
    \end{scope}

\draw[->, decorate, decoration=snake] (6.5,0) -- (8.5,0);
\node[columnlabel] at (7.5,-0.4) {\small Moving};
\node[columnlabel] at (7.5,-0.9) {\small the linear};
\node[columnlabel] at (7.5,-1.4) {\small levels to};
\node[columnlabel] at (7.5,-1.9) {\small the trunk};

     % Draw the columns first
    \foreach \x in {9,10,11,12,13,14,15} {
      \draw[black,dashed] (\x,2.25) --  (\x,-2.5);
    }
    \foreach \x/\t in {9/5,10/1,11/2,12/4,13/6,14/3} {
      \node[columnlabel] at ({\x+.5},2.75) {\t};
    }
    \foreach \x/\t in {9/L,10/L,11/L,12/L,13/B,14/B} {
      \node[columnlabel] at ({\x+.5},-2.5) {\t};
    }

    % Draw all the nodes
    \begin{scope}[
      every node/.style={treenode},
      ]
      \node (root) at (9,0) {};
      \node (1) at (10,0) {};
      \node (11) at (11,0) {};
      \node (111) at (12,0) {};
      \node (1111) at (13,0) {};
      \node (11111) at (14,1.5) {};
      \node (111111) at (15,1.75) {};
      \node (111112) at (15,1.25) {};
      \node (11112) at (14,-1.5) {};
      \node (111123) at (15,-1) {};
      \node (111124) at (15,-1.5) {};
      \node (111125) at (15,-2) {};
    \end{scope}

    % Draw all the edges
    \begin{scope}[
      ->,
      every node/.style={edgelabel},
      ]
      \draw (root) -- (1);
      \draw (1) -- (11);
      \draw[line width=.5mm] (11) -- (111);
      \draw[line width=.5mm] (111) -- (1111);
      \draw (1111) -- (11111);
      \draw (11111) -- (111111);
      \draw (11111) -- (111112);
      \draw (1111) -- (11112);
      \draw (11112) -- (111123);
      \draw (11112) -- (111124);
      \draw (11112) -- (111125);
    \end{scope}

\end{tikzpicture}
\caption{Transforming the tree of $S$.}
\label{transforming the tree of S}
\end{center}
\end{figure}
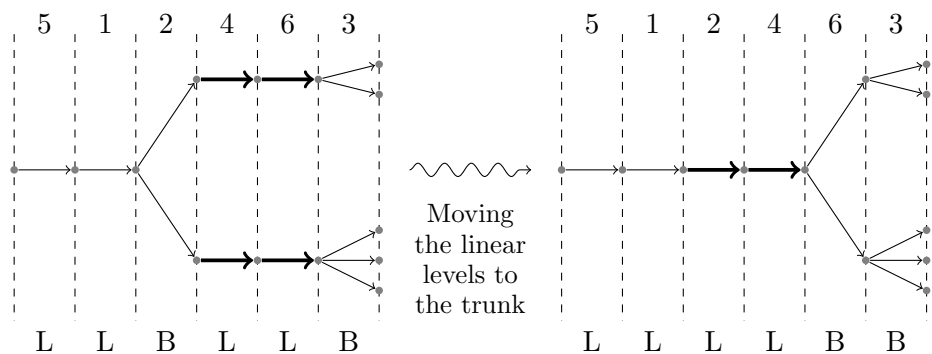

Finally, we just need to relabel the arcs and vertices of the tree we obtained in Figure~\ref{transforming the tree of S}. Figure~\ref{new tree (of a null semigroup)} shows the new labelled tree.

\begin{figure}[htb]
\begin{center}
      \begin{tikzpicture}[x=8mm, y=8mm]

    % Draw all the nodes
    \begin{scope}[
      every node/.style={treenode},
      ]
      \node (root) at (0,0) {};
      \node (1) at (1,0) {};
      \node (11) at (2,0) {};
      \node (111) at (3,0) {};
      \node (1111) at (4,0) {};
      \node (11111) at (5,1.5) {};
      \node (111111) at (6,1.75) {};
      \node (111112) at (6,1.25) {};
      \node (11112) at (5,-1.5) {};
      \node (111123) at (6,-1) {};
      \node (111124) at (6,-1.5) {};
      \node (111125) at (6,-2) {};
    \end{scope}

    \node[anchor=east] at (root) {\small $\varepsilon$};
    \node[anchor=south] at (1) {\rotatebox{90}{\small $5$}};
    \node[anchor=south] at (11) {\rotatebox{90}{\small $55$}};
    \node[anchor=south] at (111) {\rotatebox{90}{\small $555$}};
    \node[anchor=south] at (1111) {\rotatebox{90}{\small $5555$}};
    \node[anchor=south] at (11111) {\rotatebox{90}{\small $55555$}};
    \node[anchor=west] at (111111) {$555555$};
    \node[anchor=west] at (111112) {$555551$};
    \node[anchor=north] at (11112) {\rotatebox{90}{\small $55551$}};
    \node[anchor=west] at (111123) {$555515$};
    \node[anchor=west] at (111124) {$555511$};
    \node[anchor=west] at (111125) {$555512$};

    % Draw all the edges
    \begin{scope}[
      ->,
      every node/.style={edgelabel},
      ]
      \draw (root) -- node[above] {$5$} (1);
      \draw (1) -- node[above] {$5$} (11);
      \draw (11) -- node[above] {$5$} (111);
      \draw (111) -- node[above] {$5$} (1111);
      \draw (1111) -- node[anchor=south east] {$5$} (11111);
      \draw (11111) -- node[anchor=south east] {$5$} (111111);
      \draw (11111) -- node[anchor=north east] {$1$} (111112);
      \draw (1111) -- node[anchor=south west] {$1$} (11112);
      \draw (11112) -- node[anchor=south east] {$5$} (111123);
      \draw (11112) -- node[pos=.75,above] {$1$} (111124);
      \draw (11112) -- node[anchor=north east] {$2$} (111125);
    \end{scope}

  \end{tikzpicture}
\caption{New tree obtained from the tree of $S$.}
\label{new tree (of a null semigroup)}
\end{center}
\end{figure}
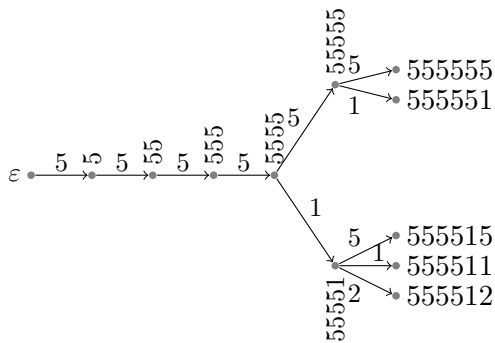

This new tree gives us the set of words
\begin{displaymath}Z=\{555555, 555551, 555515, 555511, 555512\}\subseteq \{1,2,3,4,5,6\}^*.\end{displaymath}
Using the words from $Z$ and the order $5,1,2,4,6,3$, we obtain the transformations
\begin{displaymath}\left(\begin{array}{cccccc} 1&2&3&4&5&6 \\ 5&5&5&5&5&5\end{array}\right), \left(\begin{array}{cccccc}1&2&3&4&5&6 \\ 5&5&1&5&5&5\end{array}\right), \left(\begin{array}{cccccc}1&2&3&4&5&6 \\ 5&5&5&5&5&1\end{array}\right)\end{displaymath}
\begin{displaymath}
\left(\begin{array}{cccccc}1&2&3&4&5&6 \\ 5&5&1&5&5&1\end{array}\right), \left(\begin{array}{cccccc}1&2&3&4&5&6 \\ 5&5&2&5&5&1\end{array}\right).
\end{displaymath}
We can easily check that the product of any two transformations is equal to the first transformation, which is the zero of this new semigroup. Hence we obtained a null subsemigroup of $\mathcal{T}_6$ with as many elements as $S$.
\end{example}

As a consequence of Theorem~\ref{|S|<=xi(n), zero rank 1} we know that the maximum size of a commutative nilpotent subsemigroup of $\mathcal{T}(X)$ whose zero has rank $1$ is $(|X|)\xi$. Theorem~\ref{|S|<xi(n), zero rank > 1} complements this result by examining the size of these semigroups when the zero has rank at least $2$. 

\begin{theorem}\label{|S|<xi(n), zero rank > 1}
Let $S$ be a commutative nilpotent subsemigroup of $\mathcal{T}(X)$ whose zero has rank at least $2$. Then
\begin{enumerate}
    \item If $|X|=2$, then $|S|=(|X|)\xi=(2)\xi=1$ and $S$ is a trivial semigroup. Therefore $S$ is a null semigroup;
    \item If $|X|\geqslant 3$, then $|S|<(|X|)\xi$. 
\end{enumerate}
\end{theorem}

\begin{proof}
Assume that $e$ is the zero of $S$.

Suppose that $|X|=2$. Then $X=\im e$ and, by Proposition~\ref{property nilpotent semigroups}(1), $x\beta=x=xe$ for all $x\in X$ and $\beta\in S$. Therefore $S=\{e\}$ and the result follows.

Now suppose that $|X|\geqslant 3$. Let $x\in \im e$. Define $I_x=xe^{-1}$. By Proposition~\ref{property nilpotent semigroups}, $I_x\beta\subseteq I_x$ for all $\beta \in S$ and, consequently, $\beta|_{I_x} \in \mathcal{T}(I_x)$ for all $\beta \in S$. It is easy to see that $S_x=\{\beta|_{I_x}: \beta \in S\}$ is a commutative nilpotent subsemigroup of $\mathcal{T}(I_x)$ whose zero is $e|_{I_x}$, which has rank $1$. Therefore, by Theorem~\ref{|S|<=xi(n), zero rank 1}, $|S_x|\leqslant (|I_x|)\xi$.

Let $\varphi: S \longrightarrow \prod_{x\in \im e} S_x$ be the map which sends $\beta \in S$ to the tuple whose $x$-th component is $\beta|_{I_x}$. We are going to prove that $\varphi$ is injective. Let $\beta, \gamma \in S$ be such that $(\beta)\varphi=(\gamma)\varphi$. This implies that $\beta|_{I_x}=(x)(\beta)\varphi=(x)(\gamma)\varphi=\gamma|_{I_x}$ for all $x\in\im e$ and, consequently, $\beta=\gamma$ (because $\{I_x\}_{x\in\im e}$ is a partition of $X$).

We have
\begin{align*}
|S|&\leqslant \prod_{x\in \im e} |S_x|& \textrm{[because } \varphi \textrm{ is injective]}\\
&\leqslant \prod_{x\in \im e} (|I_x|)\xi& \textrm{[by Theorem~\ref{|S|<=xi(n), zero rank 1}]}\\
&\leqslant \left(\left(\sum_{x\in \im e} |I_x|\right)-|\im e|+1\right)\xi \kern -8.7mm & \textrm{[by iterated use of Lemma~\ref{inequalities xi}(2)]}\\
&= (|X|-|\im e|+1)\xi& \textrm{[because } \{I_x\}_{x\in\im e} \textrm{ is a partition of } X \textrm{]}\\
&\leqslant (|X|-1)\xi& \textrm{[because } |\im e|\geqslant 2 \textrm{ and by Lemma~\ref{inequalities xi}(1)]}\\
&< (|X|)\xi,& \textrm{[by Lemma~\ref{inequalities xi}(1)]}
\end{align*}
which proves (2).
\end{proof}

Theorem~\ref{|S|<=xi(n), zero rank 1} guarantees that, for each commutative nilpotent subsemigroup of $\mathcal{T}(X)$ whose zero has rank $1$, there exists a null subsemigroup of $\mathcal{T}(X)$ of the same size. Note that, as a consequence of Theorem~\ref{|S|<xi(n), zero rank > 1}, this is also true when the zero of the semigroup has rank at least $2$.

Theorems~\ref{|S|<=xi(n), zero rank 1} and \ref{|S|<xi(n), zero rank > 1} assert that the maximum size of a commutative nilpotent subsemigroup of $\mathcal{T}(X)$ is $(|X|)\xi$. The next theorem describes these semigroups of size $(|X|)\xi$ and shows that they are precisely the null semigroups described in Theorem~\ref{null semigroups of maximum size}.

\begin{theorem} \label{Commutative nilpotent semigroups of maximum size are null semigroups}
Let $S$ be a commutative nilpotent subsemigroup of $\mathcal{T}(X)$ of maximum size. Then $S$ is a null semigroup.
\end{theorem}

\begin{proof}
Let $n=|X|$. According to Theorems~\ref{|S|<xi(n), zero rank > 1} and \ref{|S|<=xi(n), zero rank 1}, the maximum size of a commutative nilpotent subsemigroup of $\mathcal{T}(X)$ is $(n)\xi$. Thus $|S|=(n)\xi$.

Suppose that the zero of $S$ has rank at least $2$. Then, by Theorem~\ref{|S|<xi(n), zero rank > 1}, $n=2$ and $S$ is a null semigroup.

Suppose now that the zero of $S$ has rank $1$. We are going to use the proof of Theorem~\ref{|S|<=xi(n), zero rank 1} to prove the result. Let $\{A_i\}_{j=0}^{k}$ be the $S$-partition of $X$ and consider the order $x_1,\ldots,x_n$ of the elements of $X$ used to construct the tree of $S$.

Let $N$ be the null subsemigroup of $\mathcal{T}(X)$ obtained from $S$ by modifying the tree of $S$. Let $T_S$ and $T_N$ be the trees of $S$ and $N$, respectively. Since $|S|=(n)\xi$ then $|N|=(n)\xi$, which implies (by Theorem~\ref{null semigroups of maximum size}) that
$$N=\{\beta \in \mathcal{T}(X): \{x_1,\ldots,x_{(n)\alpha}\}\beta=\{x_1\} \textrm{ and } \im \beta\subseteq\{x_1,\ldots,x_{(n)\alpha}\}\}.$$
Thus $T_N$ has a trunk with $(n)\alpha$ arcs and the starting vertices of the arcs of the levels $x_{(n)\alpha+1},\ldots,x_n$ have outdegree $(n)\alpha$ (which means that a branching with $(n)\alpha$ arcs occurs at those vertices). Also, all the linear levels of $T_N$ are in the trunk, which means $T_N$ has $(n)\alpha$ linear levels (the levels $x_1,\ldots,x_{(n)\alpha}$) and $n-(n)\alpha$ branching levels (the levels $x_{(n)\alpha+1},\ldots,x_n$).

In order to obtain $T_N$ from $T_S$, the only thing we do (besides changing the labels of the arcs and renaming the vertices) is move the linear levels of $T_S$, that are not in the trunk, to the trunk of the tree (assuming that there are any linear levels outside the trunk of $T_S$). This means that, in the process of transforming the tree $T_S$ into the tree of $T_N$, we do not change the number of linear levels. Therefore $T_S$ and $T_N$ have the same number of linear levels, which is equal to $(n)\alpha$.

Assume, with the aim of obtaining a contradiction, that there is at least one linear level outside the trunk of $T_S$. This implies that the number of arcs of the trunk of $T_S$ is less than $(n)\alpha$. Consider the branching closest to the root of the tree, that is, the branching at the end of the trunk. By Lemma~\ref{branching preceded by s linear levels} the number of linear levels that occur before a branching is not smaller than the number of arcs of that branching. Hence the number of arcs of the branching at the end of the trunk of $T_S$ is at most $(n)\alpha-1$. When we transform $T_S$ into $T_N$, we do not change the number of arcs of that branching (we just move linear levels to the trunk). This implies that the branching closest to the root of the tree $T_N$ has less than $(n)\alpha$ arcs, which is a contradiction. Thus, all the linear levels of $T_S$ are associated with its trunk. Consequently, we do not need to move any of the linear levels of $T_S$, which means that the structure of $T_S$ is equal to the structure of $T_N$ (that is, $T_S$ and $T_N$ are equal except, possibly, the labels of the arcs and vertices).

We now know that the trunk of $T_S$ has $(n)\alpha$ arcs, all the linear levels of $T_S$ correspond to its trunk, and a branching with $(n)\alpha$ arcs occurs in all vertices that are words of length between $(n)\alpha$ and $n-1$. The arcs of the trunk are associated with the elements of $A_0\cup A_1$ and are all labelled $x_1$ (the linear levels associated with the trunk are the levels $x_1,\ldots,x_{(n)\alpha}$). Every branching has $(n)\alpha$ arcs and their labels correspond to the levels with which they are associated with. According to Lemma~\ref{i_1,...,i_s<i, i_2,...,i_s linear levels}, $(n)\alpha-1$ of those levels are linear and there is an extra level that precedes them. The only linear levels in $T_S$ are the levels $x_1,\ldots,x_{(n)\alpha}$, which means that the $(n)\alpha$ arcs of each branching are labelled with $x_1,\ldots,x_{(n)\alpha}$. Hence
\begin{displaymath}S=\left\{\beta \in \mathcal{T}(X): \{x_1,\ldots,x_{(n)\alpha}\}\beta=\{x_1\} \textrm{ and } \im \beta\subseteq\{x_1,\ldots,x_{(n)\alpha}\}\right\}=N\end{displaymath}
and, consequently, $S$ is a null semigroup.
\end{proof}

\bibliography{commutative_nilpotent_semigroups} %\jobname
\bibliographystyle{alpha}

\end{document}